\documentclass[11pt,reqno]{amsart}

\usepackage{graphicx}
\usepackage[top=20mm, bottom=20mm, left=30mm, right=30mm]{geometry}
\usepackage[colorlinks=true,citecolor=black,linkcolor=black,urlcolor=magenta]{hyperref}

\usepackage{fancyvrb}
\usepackage{ulem} 
\usepackage[T1]{fontenc} 
\usepackage{pbsi} 
\usepackage{url} 
\usepackage{calc} 
\usepackage{fancyvrb}

\usepackage{enumitem}
\setlist[itemize]{leftmargin=0.35in}

\usepackage{diagbox}

\newcommand{\QBinomial}[3]{\gkpSI{#1}{#2}_{#3}} 
\newcommand{\QPochhammer}[3]{\left(#1; #2\right)_{#3}} 
\newcommand{\gkpSI}[2]{\ensuremath{\genfrac{\lbrack}{\rbrack}{0pt}{}{#1}{#2}}} 
\newcommand{\gkpSII}[2]{\ensuremath{\genfrac{\lbrace}{\rbrace}{0pt}{}{#1}{#2}}} 
\newcommand{\Iverson}[1]{\ensuremath{\left[#1\right]_{\delta}}} 

\newcommand{\cf}[0]{\textit{cf}.\ } 
\newcommand{\citep}{\cite} 
\renewcommand{\emph}[1]{\textit{#1}}

\DeclareMathOperator{\ab}{ab}
\DeclareMathOperator{\Conv}{Conv} 
\DeclareMathOperator{\ConvP}{P}
\DeclareMathOperator{\ConvQ}{Q}
\DeclareMathOperator{\Log}{Log}

\theoremstyle{plain}
\newtheorem{theorem}{Theorem}[section]
\newtheorem{lemma}[theorem]{Lemma}
\newtheorem{claim}[theorem]{Claim}
\newtheorem{cor}[theorem]{Corollary}
\newtheorem{prop}[theorem]{Proposition}

\theoremstyle{definition}
\newtheorem{definition}[theorem]{Definition}

\theoremstyle{remark}
\newtheorem{remark}[theorem]{Remark}

\numberwithin{table}{section} 

\title[Continued Fractions and Lambert Series Generating Functions]{ 
     Continued Fractions and $q$-Series Generating Functions for the 
     Generalized Sum-of-Divisors Functions} 

\author[Maxie D. Schmidt]{Maxie D. Schmidt \\ \\ 
        School of Mathematics \\
        Georgia Institute of Technology \\ 
        117 Skiles Building \\
        686 Cherry Street NW \\ 
        Atlanta, GA 30332 \\ \\ 
        \href{mailto:maxieds@gmail.com}{maxieds@gmail.com}
        } 
\address{School of Mathematics \\ 
         Georgia Institute of Technology \\ 
         Atlanta, GA 30332
        } 
\email{maxieds@gmail.com}

\date{2017.05.04}


\allowdisplaybreaks

\begin{document}

\begin{abstract}
We construct new continued fraction expansions of Jacobi-type J-fractions in $z$ whose 
power series expansions generate the ratio of the $q$-Pochhamer symbols, 
$(a; q)_n / (b; q)_n$, for all integers $n \geq 0$ and where $a,b,q \in \mathbb{C}$ 
are non-zero and defined such that $|q| < 1$ and $|b/a| < |z| < 1$. 
If we set the parameters 
$(a, b) := (q, q^2)$ in these generalized series expansions, 
then we have a corresponding J-fraction enumerating the 
sequence of terms $(1-q) / (1-q^{n+1})$ over all integers $n \geq 0$. 
Thus we are able to define new $q$-series expansions which correspond to the 
Lambert series generating the divisor function, $d(n)$, when we set 
$z \mapsto q$ in our new J-fraction expansions. By repeated differentiation 
with respect to $z$, we also use these generating functions to formulate new 
$q$-series expansions of the generating functions for the sums-of-divisors 
functions, $\sigma_{\alpha}(n)$, when $\alpha \in \mathbb{Z}^{+}$. 
To expand the new $q$-series generating functions for these 
special arithmetic functions we define a generalized classes of so-termed 
Stirling-number-like ``$q$-coefficients'', or Stirling $q$-coefficients, 
whose properties, relations to 
elementary symmetric polynomials, and relations to the 
convergents to our infinite J-fractions are also explored within the 
results proved in the article. 
\end{abstract} 

\maketitle

\section{Introduction} 

\subsection{Continued fraction expansions of ordinary generating functions} 
\label{subSection_JFracExps_of_OGFs} 

\subsubsection*{Expansions of Jacobi-type J-fractions} 

\emph{Jacobi-type continued fractions}, or \emph{J-fractions}, 
correspond to power series defined by 
infinite continued fraction expansions of the form\footnote{ 
     \underline{Conventions}: 
     We adopt a hybrid of the notation for the implicit continued fraction 
     sequences $a_{h-1} b_h :\mapsto \ab_h$ from Flajolet's article 
     \citep{FLAJOLET80B}. Our usage of $P / Q$ to denote the 
     convergent function ratios is also consistent with the conventions from 
     this reference. 
} 
\begin{align} 
\label{eqn_J-Fraction_Expansions} 
J_{\infty}\left(z\right) 
     & = 
     \cfrac{1}{1-c_1z-\cfrac{\ab_2 z^2}{1-c_2z- 
     \cfrac{\ab_3 z^2}{\cdots}}} \\ 
\notag 
     & = 
     1 + c_1 z + \left(\ab_2+c_1^2\right) z^2 + 
     \left(2\ab_2 c_1+c_1^3+\ab_2 c_2\right) z^3 \\ 
     & \phantom{=1\ } + 
     \left(\ab_2^2+\ab_2\ab_3+3\ab_2 c_1^2 + c_1^4 + 2\ab_2 c_1c_2 + \ab_2 c_2^2 
     \right) z^4 + 
     \cdots, 
\end{align} 
for arbitrary, application-specific implicit sequences 
$\{ c_i \}_{i=1}^{\infty}$ and $\{ \ab_i \}_{i=2}^{\infty}$, and 
some typically formal series variable $z \in \mathbb{C}$ 
\citep[\cf \S 3.10]{NISTHB} \citep{WALL-CFRACS}. 
The formal series enumerated by special cases of the truncated and infinite 
J-fraction series of this form include typically divergent \emph{ordinary} 
(as opposed to typically closed-form \emph{exponential}) 
\emph{generating functions} 
for many one and two-index combinatorial sequences including the 
so-termed ``\emph{square series}'' functions studied in the references 
and in the results from Flajolet's articles 
\citep{FLAJOLET80B,FLAJOLET82,SQSERIES-CFRACS}. 

\subsubsection*{Generalized properties of the convergents to infinite J-fractions} 

We define the $h^{th}$ convergent functions, 
$\Conv_h(z) := \ConvP_h(z) / \ConvQ_h(z)$, to the infinite J-fraction in 
\eqref{eqn_J-Fraction_Expansions} recursively through the 
component numerator and denominator functions given by\footnote{ 
     \underline{Special Notation}: 
     \emph{Iverson's convention} compactly specifies 
     boolean-valued conditions and is equivalent to the 
     \emph{Kronecker delta function}, $\delta_{i,j}$, as 
     $\Iverson{n = k} \equiv \delta_{n,k}$. 
     Similarly, $\Iverson{\mathtt{cond = True}} \equiv 
     \delta_{\mathtt{cond}, \mathtt{True}} \in \{0, 1\}$, 
     which is $1$ if and only if \texttt{cond} is true, 
     in the remainder of the article. 
} 
\begin{align} 
\label{eqn_ConvFn_PhzQhz_rdefs} 
\ConvP_h(z) & = (1-c_{h} \cdot z) \ConvP_{h-1}(q, z) - 
     \ab_{h} \cdot z^2 \ConvP_{h-2}(q, z) + \Iverson{h = 1} \\ 
\notag 
\ConvQ_h(z) & = (1-c_{h} \cdot z) \ConvQ_{h-1}(q, z) - 
     \ab_{h} \cdot z^2 \ConvQ_{h-2}(q, z) + 
     (1-c_{1} \cdot z) \Iverson{h = 1} + \Iverson{h = 0}. 
\end{align} 
If we let $j_n := [z^n] J_{\infty}(z)$ in \eqref{eqn_J-Fraction_Expansions}, the 
convergents to the full J-fraction defined as above 
provide $2h$-order accurate truncated 
power series approximations to the infinite-order J-fraction generating functions
in the following form for each $h \geq 1$: 
\begin{align*}
\Conv_h(z) & = j_0 + j_1 z + j_2 z^2 + \cdots + j_{2h-1} z^{2h-1} + 
     \sum_{n \geq h} \bar{j}_{h,n} z^n. 
\end{align*} 
The rationality of $\Conv_h(z)$ for all $h \geq 1$ is a key property of these 
approximations to the infinite J-fraction expansion in 
\eqref{eqn_J-Fraction_Expansions} in that it allows us to prove finite difference 
equations of order less than or equal to $h$ for the coefficients generated by the 
power series expansions of these functions. We use these new $h$-order 
finite difference equations to prove our main theorem stated in the next 
subsection in Section \ref{Section_ProofsOfMainResults}. 
We now focus on the task of formulating specific cases of these generalized 
J-fraction expansions and convergent functions which lead to generating functions 
enumerating special number theoretic arithmetic functions we wish to study through 
these expansions. 

\subsection{Constructions of the J-fraction 
            generating function for the divisor function} 

\subsubsection*{Approach in the article} 

In this article we define the sequences implicit to the expansions of 
\eqref{eqn_J-Fraction_Expansions} to be functions of our primary series variable $q$ 
and then construct and prove new forms of convergent infinite J-fraction 
expansions whose power series expansions in $q$ after $z \mapsto q$ generate the 
\emph{divisor function}, $d(n)$, or $\sigma_0(n) = \sum_{d|n} 1$. We prove that this 
infinite J-fraction expansion corresponds to an infinite $q$-series expansion 
depending on $z$ and $q$ which may be differentiated termwise with respect to $z$ 
to obtain new modified $q$-series forms of generating functions for the 
generalized \emph{sums-of-divisors functions}, 
$\sigma_{\alpha}(n) = \sum_{d|n} d^{\alpha}$ for 
$\alpha \in \mathbb{Z}^{+}$, when we again set $z \mapsto q$ in these expansions. 
These multiplicative functions are of great interest in number theory with 
applications in many famous open problems and number theoretic 
conjectures. 
We briefly touch on the significance, applications, and relations to recent 
research in number theory of our new results in Section \ref{Section_Sig_and_Apps}. 

Within the scope of this article, we study the expansions of 
generalized divisor function 
generating functions in the context of a much more broadly applicable method for 
enumerating sequences of special functions based on 
generalized continued fractions. 
In particular, we find formulas $\ab_i(q)$ and $c_i(q)$ 
for the sequences, $\langle \ab_i \rangle$ and $\langle c_i \rangle$, in 
\eqref{eqn_J-Fraction_Expansions} such that that the infinite J-fraction, 
denoted by $J_{\infty}(q, z)$ when the component sequences in this case 
are clear from context, generates the terms 
$[z^n] J_{\infty}(z) \equiv [z^n] J_{\infty}(q, z) = \frac{1-q}{1-q^{n+1}}$ 
for all $n \geq 0$. 
We then see from this construction that the power series expansions of this 
scaled J-fraction in $q$ when $z \mapsto q$ generates the divisor function, 
$\sigma_0(n) \equiv d(n)$ as 
\begin{align*} 
\sum_{n \geq 1} \frac{q^n}{1-q^n} & = \sum_{m \geq 1} d(m) q^m,\ |q| < 1. 
\end{align*} 
Moreover, this procedure allows us to use the 
properties of more general continued fraction results, of which J-fractions are 
a special case, to formulate an infinite $q$-series in $q$ and $z$ for this 
generating function whose $h^{th}$ partial sum is $2h$-order accurate in 
enumerating the terms of $(1-q^{n+1})^{-1}$ over $z$. 

\subsubsection*{Complexity and structure of the new $q$-series 
                generating function expansions} 

We also notice that by the nature and complexity of the special arithmetic 
functions that we are able to generate with these constructions proved in the 
article, a priori these sums do (and should) correspond to complicated and 
rather involved combinatorial objects. 
We first demonstrate by a triple of examples 
the tricky nature inherent to directly expanding the symbolic sequences in the 
definition of \eqref{eqn_J-Fraction_Expansions} whose \emph{Lambert series}
generating functions in $q$ when $z \mapsto q$ generate the special functions, 
$d(n)$, $\sigma_1(n)$, and $\sigma_{\alpha}(n)$. 
More precisely, for $[z^n] J_{\infty}(z) := 1 / (1-q^n)$, we see that 
\begin{align*} 
c_1(q) & = \frac{1}{1-q} \\ 
c_2(q) & = \frac{1+q+4q^2}{2(q^3-1)} \\ 
c_3(q) & = \frac{1+5q+14q^2+26q^3+34q^4+25q^5+9q^6}{ 
     2 (1+q+q^2) (1+2q+3q^2) (1+q+q^2+q^3+q^4)} \\ 
\ab_2(q) & = -\frac{2q}{(1-q)^2 (1+q)} \\ 
\ab_3(q) & = -\frac{(1-q) (1+2q+3q^2)}{4(1+q) (1+q^2) (1+q+q^2)}, 
\end{align*} 
and for $[z^n] J_{\infty}(z) := n / (1-q^n)$, we similarly compute that 
\begin{align*} 
c_1(q) & = \frac{1}{1-q} \\ 
c_2(q) & = \frac{q(-1-q+8q^2)}{(1-q)(1-3q)(1+q+q^2)} \\ 
c_3(q) & = -\frac{1-5 q^2-16 q^3-16 q^4+40 q^5+136 q^6+144 q^7+67 q^8+8 q^9+q^{10}}{ 
     (1-3 q) \left(1+q+q^2\right) \left(1+q+q^2+q^3+q^4\right) 
     \left(-1+4 q^2+8 q^3+q^4\right)} \\ 
\ab_2(q) & = \frac{1-3q}{(1-q)^2 (1+q)} \\ 
\ab_3(q) & = \frac{(1-q)^3 (-1+4q^2+8q^3+q^4)}{(1+q) (1-3q)^2 (1+q^2) (1+q+q^2)^2}. 
\end{align*} 
More generally, we can see that there is no apparent special formula for the 
expansions of the implicit sequences for 
$[z^n] J_{\infty}(z) := n^{\alpha} / (1-q^n)$ by computing that 
\begin{align*} 
c_1(q) & = \frac{1}{1-q} \\ 
c_2(q) & = \frac{3^{\alpha } (-1+q)^2 (1+q)+\left(1+2^{1+\alpha } (-1+q)+q\right) \left(1+q+q^2\right)}{\left(1+2^{\alpha } (-1+q)+q\right) \left(-1+q^3\right)} \\ 
c_3(q) & = \scriptstyle{\Biggl[\frac{2^{1+3 \alpha } 3^{\alpha }}{(-1+q)^3 (1+q)^2 \left(1+q^2\right) \left(1+q+q^2\right)}-\frac{27^{\alpha }}{\left(-1+q^3\right)^3}+\frac{5^{\alpha }}{(-1+q)^4 \left(1-q^5\right)}+\frac{20^{\alpha }}{\left(-1+q^2\right)^2 \left(1-q^5\right)}+\frac{\frac{2^{1+3 \alpha }}{(1+q)^2 \left(1+q^2\right)}+\frac{9^{\alpha }}{\left(1+q+q^2\right)^2}}{(-1+q)^5}} \\ 
     & \phantom{=\Biggl[\ } + 
     \scriptstyle{\frac{2^{\alpha } \left(-\frac{2\ 9^{\alpha } (-1+q) (1+q)^3}{\left(1+q+q^2\right)^2}+\frac{8^{\alpha } \left(-1+3 q^2\right)}{1+q^2}\right)}{(-1+q)^5 (1+q)^4}+\frac{2^{\alpha } \left(\frac{2^{\alpha } 3^{1+\alpha }}{\left(-1+q^2\right)^2 \left(1-q^3\right)}-\frac{2\ 5^{\alpha }}{1-q^2-q^5+q^7}+\frac{2^{1+\alpha } 3^{\alpha }}{1+q^3 \left(-1-q+q^4\right)}\right)}{(-1+q)^2}\Biggr]} \Biggl/ \\ 
     & \phantom{=\Biggl[\ } 
     \scriptstyle{\Biggl[\left(\frac{1}{(-1+q)^2}+\frac{2^{\alpha }}{-1+q^2}\right) \left(-\frac{8^{\alpha }}{\left(-1+q^2\right)^3}+\frac{9^{\alpha }}{\left(-1+q^3\right)^2}+\frac{4^{\alpha }}{(-1+q)^2 \left(1-q^4\right)}+\frac{2^{\alpha } \left(-\frac{2\ 3^{\alpha }}{(-1+q)^2 \left(1+q+q^2\right)}+\frac{4^{\alpha }}{-1+q^4}\right)}{1-q^2}\right)\Biggr]} \\ 
\ab_2(q) & = -\frac{1}{(-1+q)^2}+\frac{2^{\alpha }}{1-q^2} \\ 
\ab_3(q) & = \frac{\frac{8^{\alpha }}{\left(-1+q^2\right)^3}+\frac{8^{\alpha }}{\left(-1+q^2\right)^2 \left(1+q^2\right)}-\frac{2^{1+\alpha } 3^{\alpha }}{(-1+q)^3 (1+q) \left(1+q+q^2\right)}-\frac{9^{\alpha }}{\left(-1+q^3\right)^2}+\frac{4^{\alpha }}{(-1+q)^3 \left(1+q+q^2+q^3\right)}}{\left(\frac{1}{(-1+q)^2}+\frac{2^{\alpha }}{-1+q^2}\right)^2}. 
\end{align*} 
We can also compute numerical sequences directly generating the functions, 
$\sigma_{\alpha}(n)$, over $z$, i.e., in place of attempting to indirectly 
enumerate these special sequences by generating their corresponding 
Lambert series function expansions through \eqref{eqn_J-Fraction_Expansions}. 
However, the resulting sequence expansions in these cases are similarly 
complicated and unreliable. 

\subsubsection*{An intermediate approach} 

We observe that we may bypass the seemingly complex forms of the 
implicit J-fraction sequences characteristic of empirically computing the 
first terms of these sequences whose corresponding $q$-series expansion of 
\eqref{eqn_J-Fraction_Expansions} generates the 
divisor sum functions, $d(n)$ and $\sigma_{\alpha}(n)$ for 
$\alpha \in \mathbb{Z}^{+}$, using a trick and a special case. 
That is, by constructing a new class of these J-fraction expansions generating the 
ratios of the $q$-Pochhammer symbols, $(a; q)_n / (b; q)_n$, 
according to Definition \ref{def_ParamSeqDefs_and_Notations} below 
for some fixed 
non-zero $a, b \in \mathbb{C}$, we may generate the divisor function generating 
function terms defined above in the particular case where $(a, b) := (q, q^2)$. 

\begin{definition}[Sequence Definitions and Notation] 
\label{def_ParamSeqDefs_and_Notations} 
We define the two sequences implicitly defined by \eqref{eqn_J-Fraction_Expansions} 
in the next theorem for some fixed non-zero $a, b, q \in \mathbb{C}$ to be 
\begin{align} 
\label{eqn_seq_defs_ababq_cabq_thm_stmt} 
\ab_i(a, b; q) & := \frac{q^{2i-4} (1-bq^{i-3}) (1-a q^{i-2}) (a - b q^{i-2}) 
     (1-q^{i-1})}{(1-b q^{2i-5}) (1-b q^{2i-4})^2 (1-b q^{2i-3})},\ 
     \text{ for integers } i \geq 2 \\ 
\notag 
c_i(a, b; q) & := \begin{cases} 
     \frac{q^{i-2}\left(q+ab q^{2i-3} + a\left(1-q^{i-1}-q^i\right) 
     + b \left(-1-q+q^{i}\right)\right)}{ 
     (1-b q^{2i-4}) (1-b q^{2i-2})} & \text{ if } i \geq 2; \\ 
     \frac{a-1}{b-1} & \text{ if } i = 1; \\ 
     0 & \text{ otherwise, } 
     \end{cases}
\end{align} 
where the \emph{$h^{th}$ modulus} component products are given by 
\begin{align} 
\notag 
\lambda_h(a, b; q) & := \ab_2(a, b; q) \cdots \ab_{h}(a, b; q) \\ 
\label{eqn_lambda_abq_hth_modulus_products} 
     & \phantom{:} = 
     \frac{a q^{(h-1)^2} \left(b/q; q\right)_{h-1} \left(a; q\right)_{h-1} 
     \left(b/a; q\right)_{h-1} (q; q)_{h-1}}{\left(b/q; q^2\right)_{h-1} 
     \left(b; q^2\right)_{h-1}^2 \left(bq; q^2\right)_{h-1}}. 
\end{align} 
\textit{Notation.} 
We formally introduce the notation of $J_{\infty}(a, b; q, z)$ to denote the 
left-hand-side function in \eqref{eqn_J-Fraction_Expansions} and 
$\Conv_h(a, b; q, z)$ for its $h^{th}$ convergents to denote shorthands for the 
J-fraction expansions involving these sequences parametrized in the non-zero 
$a, b, q \in \mathbb{C}$. Where it is clear from context that the particular 
J-fraction defined in this form is defined in terms of definite functions of 
$q$, we drop the first two parameters in the previous notation for the 
generalized expansions to write $J_{\infty}(q, z)$ for the infinite J-fraction and 
$\Conv_h(q, z)$ for its convergents. Typically, we assume in these cases that 
$(a, b) := (q, q^2)$ when using this abbreviated notation. 
\end{definition} 

There is ``\emph{no free lunch}'' in that these comparatively simple 
sequence forms result in expansions of the convergent functions, denoted by 
$\Conv_h(a, b; q, z)$ in these cases, which inherit a more rich and complex 
structure involving paired products of the $\ab_i(a, b; q)$ 
functions interleaved with the 
expansions of a generalized set of \emph{Stirling q-coefficients} whose 
expansions are explicitly formed by \emph{elementary symmetric polynomials} over the 
sequence of $q$-functions, $c_i(a, b; q)$. 
The resulting convergent-based series for these generating functions is 
considered to have the form of a ``$q$-series expansion'' since it is 
intertwined with finite $q$-Pochhammer symbols and reciprocal paired sums of the 
Stirling number $q$-coefficients weighted by individual products of the 
functions, $\ab_i(a, b; q)$, which are parametrized in the choices of non-zero 
$a, b, q \in \mathbb{C}$. 
Choosing a construction of our new $q$-series results based on the trick to 
enumerate the special case of $[z^n] J_{\infty}(z) = (a; q)_n / (b; q)_n$ by these 
J-fractions, we do in fact obtain the desired simpler forms of the sequences, 
though again we see in turn in the next sections that the corresponding 
convergent denominator functions, $Q_h(a, b; q, z)$, in 
\eqref{eqn_ConvFn_PhzQhz_rdefs} satisfy the 
predicted more complicated and involved expansions which 
we define and prove in the next sections of the article. 

\subsubsection*{Definitions and statement of the main theorem} 

\begin{theorem}[J-fractions Generating a Special Ratio of $q$-Pochhammer Symbols] 
\label{theorem_MainGenabqz_J-FracThm} 
Let $a, b, q, z \in \mathbb{C}$ denote fixed non-zero parameters such that 
$|b/a| < |z| < 1$ and $|q| < 1$. 
We claim that for the special sequences defined in the notation of 
Definition \ref{def_ParamSeqDefs_and_Notations} that for all $h \geq 2$ we have 
\begin{align} 
\label{eqn_mainthm_stmt_part1} 
\Conv_h(a, b; q, z) & = \sum_{i=1}^h \frac{\lambda_i(a, b; q) 
     z^{2i-2}}{Q_{i-1}(a, b; q, z) Q_i(a, b; q, z)}, 
\end{align} 
where for each integer $h \geq 1$ and all $0 \leq n < 2h$ 
\begin{align} 
\label{eqn_Convh_coeffs_zn_eq_ratio}
[z^n] \Conv_h(a, b; q, z) & = \frac{(a; q)_n}{(b; q)_n}, 
\end{align} 
In our special case of interest where $(a, b) := (q, q^2)$ we have 
\begin{align} 
\label{eqn_mainthm_stmt_part3} 
J_{\infty}(q, z) & = \frac{q(1+q)}{1+q-z} + \sum_{i \geq 2} 
     \frac{q \cdot q^{(i-1)^2} \left(q; q\right)_{i-1}^4 z^{2i-2}}{ 
     \left(q; q^2\right)_{i-1} 
     \left(q^2; q^2\right)_{i-1}^2 \left(q^3; q^2\right)_{i-1} \times 
     Q_{i-1}(q, z) Q_i(q, z)} \\ 
\notag 
     & = (1-q) \times \sum_{n \geq 0} \frac{z^n}{1-q^{n+1}} z^n, 
\end{align} 
provided that this infinite J-fraction is convergent for $|q|, |qz| < 1$ and our 
choices of $a, b \neq 0$. 
\end{theorem} 

We will prove this theorem and the convergence of the limiting case of the finite 
convergent function sums for $(a, b) := (q, q^2)$ as two of our main results in 
Section \ref{Section_ProofsOfMainResults} below. 
We will first need some machinery for expanding the $h^{th}$ convergents to 
$J_{\infty}(z)$ for arbitrary sequences, $\langle \ab_i \rangle$ and 
$\langle c_i \rangle$, developed by the results proved in 
Section \ref{Section_ExpsOfGenConvFns}. 

\subsection{A comparison of J-fraction expansions for known generating functions 
            of the divisor and sum-of-divisors functions} 
\label{subSection_ExamplsOfKnownGFs}

The significance of Theorem \ref{theorem_MainGenabqz_J-FracThm} 
is that it allows us to formulate, as we will soon 
see, rich structured expansions of Lambert series generating functions for the 
divisor function, and by extension for $\sigma_{\alpha}(n)$ for all integers 
$\alpha \geq 1$. 
The reference \citep{SQSERIES-CFRACS} 
provides related convergent $q$-series expansions for the 
sum-of-divisors function, $\sigma_1(n)$, given through reciprocal sums over the 
\emph{Gaussian polynomials}, or \emph{$q$-binomial coefficients} in place of the 
sums we develop involving our so-termed generalized Stirling $q$-coefficients. 
These known special sequence generating functions are given by 
\begin{align*} 
\sum_{k=0}^{\infty} \sigma_1(2k+1) q^{2k+1} & = 
     \sum_{b = \pm 1} \frac{b}{16}\left( 
     1 + 2bq \times \sum_{i=1}^{\infty} \frac{(-1)^{i-1} (bq)^{3i(i-1)} 
     \QPochhammer{q^2}{q^2}{i-1}}{\sum\limits_{0 \leq j \leq n < 2i} 
     \QBinomial{i}{j}{q^2} \QBinomial{i-1}{n-j}{q^2} q^{2j} 
     (-b q^{2i-1})^n} \right)^{4}, 
\end{align*} 
and 
\begin{align*} 
\sum_{n \geq 1} \sigma_1(n) q^n & = -q \cdot d / dq \left[ 
      \Log\QPochhammer{q}{q}{\infty}\right], 
\end{align*} 
where for $|q| < 1$ we have that 
\begin{align*} 
\QPochhammer{q}{q}{\infty} & = 
     1 - q \times \sum_{i=1}^{\infty} 
     \frac{(-1)^{i-1} q^{(9i-2)(i-1)/2} \QPochhammer{q^3}{q^3}{i-1}}{ 
     \sum\limits_{0 \leq j \leq n < 2i} \QBinomial{i}{j}{q^3} 
     \QBinomial{i-1}{n-j}{q^3} q^{3j} \cdot q^{(3i-2) n}} \\ 
     & \phantom{=1\ } - 
     q^2 \times \sum_{i=1}^{\infty} 
     \frac{(-1)^{i-1} q^{(9i+2)(i-1)/2} \QPochhammer{q^3}{q^3}{i-1}}{ 
     \sum\limits_{0 \leq j \leq n < 2i} \QBinomial{i}{j}{q^3} 
     \QBinomial{i-1}{n-j}{q^3} q^{3j} \cdot q^{(3i-1) n}}. 
\end{align*} 
We note that an example related to our methods constructed within this 
article, but that we do not explicitly cite closed-form sums for, 
provides another $q$-series generating functions for the divisor function 
in the form of \citep{QSERIESIDENTS-DIVFNS} 
(\cf Table \ref{table_OtherqSeriesJFracExps} on page 
\pageref{table_OtherqSeriesJFracExps})
\begin{align*} 
\sum_{n=1}^{\infty} d(n) q^n & = \QPochhammer{q}{q}{\infty} \times 
     \sum_{n=1}^{\infty} \frac{n^{} q^n}{\QPochhammer{q}{q}{n}}, 
\end{align*} 
where we can generate the terms of 
$n^{} / \QPochhammer{q}{q}{n}$ through first-order derivatives of the 
series for $1 / \QPochhammer{z}{q}{\infty}$ also identified in the reference. 
Additional identities of Dilcher from his article 
provide expansions of the series coefficients of 
\[
U_{\alpha}(q) = 
     \sum_{n \geq 1} n^{\alpha} q^n \times \prod_{j=n+1}^{\infty} (1-q^j) = 
     \QPochhammer{q}{q}{\infty} \times 
     \sum_n \frac{n^{\alpha} q^n}{\QPochhammer{q}{q}{n}}, 
\]
when $\alpha \in \mathbb{Z}^{+}$. 
Unlike Dilcher's article referenced above, our methods provide a 
generalizable construction that is employed to form explicit $q$-series 
generating functions that enumerate the sequence of $\sigma_{\alpha}(n)$ for 
\emph{any} fixed integers $\alpha \geq 0$. 

\section{Expansions of generalized convergent functions to infinite J-fractions} 
\label{Section_ExpsOfGenConvFns} 

\subsection{Definitions and statements of key lemmas} 

\begin{definition}[Special Sums and Triangles of Stirling $q$-Coefficients] 
\label{def_SpSums_GenStirlingqCoeffs} 
Let the function, $S_{h,m,s}(z)$, be defined by the nested sums in the next 
equation for non-negative integers $h \geq 2$, $m \leq h$, and $s \leq mh$. 
\begin{align*} 
S_{h,m,s}(z) & := \sum_{k_1=2}^{h-2(m-1)} \sum_{k_2=k_1+2}^{h-2(m-2)} \cdots 
     \sum_{k_m=k_{m-1}+2}^{h} \left[ 
     \frac{\ab_{k_1}}{(1-c_{k_1}z)(1-c_{k_1-1}z)} \cdots 
     \frac{\ab_{k_m}}{(1-c_{k_m}z)(1-c_{k_m-1}z)}\right] \times \\ 
     & \phantom{:=\sum\sum\cdots\sum\sum\sum\qquad\ } \times 
     \Iverson{k_1+\cdots+k_m=s}
\end{align*} 
Next, we let the most general forms of the Stirling-number-like $q$-coefficient 
triangles be defined recursively by 
\begin{align} 
\label{eqn_GenStirqCoeffs_GenS1hkc} 
\gkpSI{h}{k}_c & := \gkpSI{h-1}{k}_c - c_h \gkpSI{h-1}{k-1}_c + \Iverson{h = k = 0}. 
\end{align} 
If we let the sequences implicit to the expansions of 
\eqref{eqn_J-Fraction_Expansions} correspond to the special $q$-Pochhammer ratio 
sequences defined in Definition \ref{def_ParamSeqDefs_and_Notations}, we define 
another auxiliary notation for these now so-termed Stirling $q$-coefficients 
where $\gkpSI{h}{k}_c \mapsto \gkpSI{h}{k}_{a,b,q}$ and these special case 
coefficients satisfy the corresponding triangular recurrence relation given by 
\begin{align} 
\label{eqn_GenStirqCoeffs_GenS1habq} 
\gkpSI{h}{k}_{a,b,q} & := \gkpSI{h-1}{k}_{a,b,q} - c_h(a,b; q) 
     \gkpSI{h-1}{k-1}_{a,b,q} + \Iverson{h = k = 0}. 
\end{align} 
\end{definition} 

\begin{lemma}[Products Generating the Stirling $q$-Coefficients]  
\label{lemma_ProdGFs_for_GenStirqCoeffs}
For all $h \geq 0$ and $0 \leq k \leq h$, we have that the most general forms of the 
Stirling $q$-coefficients defined above are generated by the products 
\begin{align*}
\gkpSI{h}{k}_c & = [z^k](1-c_1z)(1-c_2z) \cdots (1-c_hz) + \Iverson{h = k = 0}. 
\end{align*} 
\end{lemma} 

\begin{lemma}[Expansions of the Convergent Denominator Functions] 
\label{lemma_ExactExpsOfThe_ConvDenomFns} 
For the most general forms of the convergent denominator functions defined by 
\eqref{eqn_ConvFn_PhzQhz_rdefs} and the 
corresponding most general forms of the Stirling $q$-coefficients, 
$\gkpSI{h}{k}_c$, 
defined in this section above, we have the following expansion identities for 
all $0 \leq n \leq h$ when $h \geq 2$: 
\begin{align}
\tag{i} 
Q_h(z) & = (1-c_1z)\times\cdots\times(1-c_hz)\left[1 + 
     \sum_{m=1}^{\lfloor h/2 \rfloor} \sum_{s=0}^{mh} \left(-z^2\right)^m 
     S_{h,m,s}(z)\right] \\ 
\tag{ii} 
[z^n] Q_h(z) & = \gkpSI{h}{n}_c + 
     \sum_{m=1}^{\lfloor h/2 \rfloor} \sum_{s=0}^{mh} \sum_{k=0}^n 
     (-1)^m \gkpSI{h}{n-k}_c [z^{k-2m}] S_{h,m,s}(z). 
\end{align} 
\end{lemma}

\begin{remark} 
The intuition for the statement of (i) in 
Lemma \ref{lemma_ExactExpsOfThe_ConvDenomFns}, which we prove by induction on $h$ 
in Section \ref{subSection_ProofsOfKeyLemmas}, is an interpretation of the 
successive expansions of the recurrence relation in \eqref{eqn_ConvFn_PhzQhz_rdefs} 
for $Q_h(z)$. Namely, we have on one hand terms of the form $(1-c_i z)$ multiplied 
recursively by the inductive expansions of $Q_h(z)$, and similarly 
on the other hand, we have 
terms of the form $z^2 \cdot \ab_i$ multiplied by inductive expansions of the 
expansion forms of the denominator functions with two indices of difference. 
This implies that we may pull out a factor of the product, 
$(1-c_1z) \cdots (1-c_hz)$, from the overall expansion of $Q_h(z)$, which leaves us 
with a sum of weighted terms in products of $z^{2m} \cdot \ab_{i_1} \cdots \ab_{i_m}$ 
times corresponding paired reciprocals of $(1-c_{i_j}z)(1-c_{i_j-1}z)$ in these 
expansions. The definition of the special sums, $S_{h,m,s}(z)$, from 
Definition \ref{def_SpSums_GenStirlingqCoeffs} makes the 
exact forms of these expansions for $Q_h(z)$ precise 
according to the statement of the lemma. 
\end{remark} 

\begin{remark}[Exact Formulas for the Generalized $q$-Coefficients] 
We see immediately, and can easily verify by induction, that the 
particular first columns of the triangular $q$-coefficients defined in 
Definition \ref{def_SpSums_GenStirlingqCoeffs} are given by the formulas 
\begin{align*} 
\gkpSI{h}{0}_{q,q^2,q} & = 1 \\ 
\gkpSI{h}{1}_{q,q^2,q} & = -\frac{1}{1+q} + \sum_{k=0}^{h-2} \Biggl[ 
     \frac{q}{2(1-q^{k+2})} - \frac{q^3+2q^2-3q-2}{2(q^2-1)(1+q^{k+2})} \\ 
     & \phantom{=-\frac{1}{1+q}+\sum\Biggl[\ } - 
     \frac{1}{2(1+q)(1-q^{k+1})} - \frac{2q-3}{2(1-q)(1+q^{k+1})}\Biggr], 
\end{align*} 
where the sums on the right-hand-side of the second equation are expanded through 
\emph{$q$-polygamma functions} for each $h \geq 1$ \cite[\S 5.18]{NISTHB}. 
Formalizing exact formulas for the subsequent cases of columns indexed by 
$k \geq 2$ is a dicey proposition by inspection alone. 
That being said, the products generating 
these coefficients given by Lemma \ref{lemma_ProdGFs_for_GenStirqCoeffs} do 
in fact give these coefficients another rich structure which inherits from the 
properties of elementary symmetric polynomials with respect to a single variable. 
For example, we have that 
\begin{align*} 
\gkpSI{h}{2}_c & = \sum_{1 \leq i_1 < i_2 \leq h} c_{i_1} c_{i_2} \\ 
\gkpSI{h}{3}_c & = \sum_{1 \leq i_1 < i_2 < i_3 \leq h} c_{i_1} c_{i_2} c_{i_3} \\ 
\gkpSI{h}{k}_c & = \sum_{1 \leq i_1 < \cdots < i_k \leq h} c_{i_1} \cdots c_{i_k} ,\ 
     \text{ for integers } k \geq 1, 
\end{align*} 
and in particular, if we define 
\begin{equation*} 
S_m(c_1, \ldots, c_h) := \sum_{j=1}^h c_j^m, 
\end{equation*} 
then the \emph{Newton-Girard formulas}, or \emph{Newton's identities}, 
imply that for $0 \leq k \leq h$ 
\begin{align*} 
(-1)^k k \gkpSI{h}{k}_c + \sum_{m=1}^k (-1)^{k-m} S_m(c_1, \ldots, c_h) 
     \gkpSI{h}{m-k}_c & = 0. 
\end{align*} 
For other specific properties of these generalized Stirling-like coefficients, 
we may consult the results in the references 
\citep{CHARALAM-GENQ-STIRLINGNUMS,MANSOUR-GENSTIRNUMS,FSTIRNUMS}. 
\end{remark} 

\subsection{Proofs of the key lemmas} 
\label{subSection_ProofsOfKeyLemmas}

\begin{proof}[Proof of Lemma \ref{lemma_ProdGFs_for_GenStirqCoeffs}] 
The proof of this result is follows by defining a recurrence for the 
right-hand-side products and showing that it generates precisely the same 
triangular sequence as the generalized Stirling numbers we defined in 
Definition \ref{def_SpSums_GenStirlingqCoeffs}. 
In particular, for integers $h, k \geq 0$ we define the 
following sequence in terms of the products in the statement of the lemma: 
\begin{equation*} 
f_{h,k} := [z^k] (1-c_1z) \cdots (1-c_hz) + \Iverson{h = k = 0}. 
\end{equation*} 
Then by expanding the right-hand-side coefficients of the last equation 
recursively we obtain that 
\begin{align*} 
f_{h,k} & = [z^k]\left((1-c_1z) \cdots (1-c_{h-1}z) - c_h \cdot z 
     (1-c_1z) \cdots (1-c_{h-1}z)\right) + \Iverson{h = k = 0} \\ 
     & = 
     [z^k] (1-c_1z) \cdots (1-c_{h-1}z) - c_h \cdot [z^{k-1}] 
     (1-c_1z) \cdots (1-c_{h-1}z) + \Iverson{h = k = 0} \\ 
     & = 
     f_{h-1,k} - c_h \cdot f_{h-1,k-1} + \Iverson{h = k = 0}. 
\end{align*} 
Thus we see that the product-wise coefficients, $f_{h,k}$, satisfy precisely the 
same recurrence relation as the generalized Stirling numbers defined above 
\emph{and} have the same initial conditions. Namely, that 
$f_{h,k}, \gkpSI{h}{k}_c = 0$ when $k, h < 0$ and 
$f_{h,k}, \gkpSI{h}{k}_c = 1$ when $h = k = 0$. 
Hence, we conclude that the two sequences define the same triangle, and so the 
generalized Stirling numbers (or \emph{Stirling $q$-coefficients}) 
are generated by the right-hand-side products in the statement of the lemma. 
\end{proof} 

\begin{proof}[Proof of Lemma \ref{lemma_ExactExpsOfThe_ConvDenomFns} (i)] 
We prove part (i) of the lemma by induction on $h$. When $h = 1$, we have that 
$Q_h(z) = 1-c_1z$ and that $\lfloor h/2 \rfloor = 0$ so that the double sum on the 
right-hand-side of (i) is zero, which implies that (i) is true in this case. 
Suppose that $h \geq 2$ and that (i) is true as a function of $k$ 
for all $k < h$. Then by expanding 
\eqref{eqn_ConvFn_PhzQhz_rdefs} according to our inductive hypothesis, we obtain 
\begin{align*} 
Q_h(z) & = (1-c_1z) \cdots (1-c_hz) \times \Biggl[1 + 
     \sum_{m=1}^{\left\lfloor \frac{h-1}{2} \right\rfloor} \sum_{s=0}^{m(h-1)} 
     (-1)^m z^{2m} S_{h-1,m,s}(z) \\ 
     & \phantom{\Biggl[\ } - 
     \frac{z^2 \ab_h}{(1-c_{h-1}z)(1-c_h z)} - 
     \frac{z^2 \ab_h}{(1-c_{h-1}z)(1-c_h z)} 
     \sum_{m=1}^{\left\lfloor \frac{h-2}{2} \right\rfloor} \sum_{s=0}^{m(h-2)} 
     (-1)^m z^{2m} S_{h-2,m,s}(z) \Biggr] 
\end{align*} 
We first suppose that $h := 2j+1$ is odd and will give an instructive proof of the 
formula in this case which similarly leads to a proof of the formula in the case 
where $h := 2j$ is even. 
We next define the auxiliary function, $T_h(z)$, and expand its terms as follows: 
\begin{align*} 
T_h(z) & = 
     \sum_{m=1}^{\lfloor h/2 \rfloor} \sum_{s=0}^{mh} 
     (-1)^m z^{2m} S_{h,m,s}(z) - 
     \sum_{m=1}^{\lfloor (h-1)/2 \rfloor} \sum_{s=0}^{mh} 
     (-1)^m z^{2m} S_{h-1,m,s}(z) \\ 
     & \phantom{:} = 
     \underset{:= \widetilde{T}_h(z) 
     \text{ contains all terms with a factor of } \ab_h }{\underbrace{ 
     \sum_{m=1}^{\lfloor h/2 \rfloor} \sum_{s=0}^{m(h-1)} 
     (-1)^m z^{2m} \left[S_{h,m,s}(z) - S_{h-1,m,s}(z)\right]}} + 
     \underset{= -\frac{z^2 \ab_h}{(1-c_{h-1}z)(1-c_hz)} 
     \text{ for $(m, s) = (1, h)$ }}{\underbrace{ 
     \sum_{m=1}^{\lfloor h/2 \rfloor} \sum_{s=0}^{m-1} 
     (-1)^m z^{2m} S_{h,m,s+mh-m+1}(z)}}. 
\end{align*} 
To simplify notation, let the shorthand 
$d_i := \ab_i / \left[(1-c_{i-1}z)(1-c_iz)\right]$. 
All terms with a factor of $\ab_h$ in the expansion of $\widetilde{T}_h(z)$ 
are of the form $d_{i_1} \times \cdots \times d_{i_{m-1}} \cdot d_h$ where by the 
definition of the special sums, $S_{h,m,s}(z)$, we must have that 
$i_{m-1} \leq h-2$ and $i_j \in [i_{j-1}+2, i_{j+1}-2]$ for $1 \leq j < m$ and 
where by convention we set $i_0 := 0$. Equivalently, we see that 
\begin{align*} 
\widetilde{T}_h(z) & = \sum_{k_1=2}^{h-2-2(m-1)} \cdots 
     \sum_{k_{m-1}=k_{m-2}+2}^{h-2} \sum_{k_m=h}^h d_{k_1} \cdots d_{k_m} \\ 
     & = 
     \sum_{m=2}^{\lfloor h/2 \rfloor} \sum_{s=0}^{m(h-2)} 
     d_h \times (-1)^m z^{2m} S_{h-2,m-1,s}(z) \\ 
     & = -\sum_{m=1}^{\left\lfloor \frac{h-2}{2} \right\rfloor} \sum_{s=0}^{m(h-2)} 
     z^2 d_h \times (-1)^mz^{2m} S_{h-2,m,s}(z), 
\end{align*} 
where the upper index $m(h-2)$ of the second sum is formed by considering the 
maximum possible sum of $k_1+\cdots+k_m=s$. 
Then we finally see that 
\begin{align*} 
\sum_{m=1}^{\lfloor h/2 \rfloor} \sum_{s=0}^{mh} (-1)^m z^{2m} S_{h,m,s}(z) & = 
     -z^2 d_h + 
     \sum_{m=1}^{\lfloor (h-1)/2 \rfloor} \sum_{s=0}^{m(h-1)} 
     (-1)^m z^{2m} S_{h-1,m,s}(z) \\ 
     & \phantom{=pz^2 d_h\ } - 
     z^2 d_h \times \sum_{m=1}^{\lfloor (h-2)/2 \rfloor} \sum_{s=0}^{m(h-2)} 
     (-1)^m z^{2m} S_{h-2,m,s}(z), 
\end{align*} 
and so the result is proved. 
\end{proof} 

\begin{proof}[Proof of Lemma \ref{lemma_ExactExpsOfThe_ConvDenomFns} (ii)] 
The result in (ii) of the lemma is not difficult to obtain from the statement of the 
first result in (i). In particular, we may first rewrite (i) in the following forms 
resulting from the expansions in Lemma \ref{lemma_ProdGFs_for_GenStirqCoeffs}: 
\begin{align*} 
Q_h(z) & = \sum_{k=0}^h \gkpSI{h}{k}_c z^k + \sum_{m=1}^{\lfloor h/2 \rfloor} 
     \sum_{s=0}^{mh} \sum_{s_0=0}^h (-1)^m z^{s_0} 
     [z^{s_0-k}]\left(\sum_{k=0}^h \gkpSI{h}{k}_c z^k\right) \times 
     [z^{k-2m}] S_{h,m,s}(z) \\ 
     & = 
     \sum_{k=0}^h \gkpSI{h}{k}_c z^k + \sum_{m=1}^{\lfloor h/2 \rfloor} 
     \sum_{s=0}^{mh} \sum_{s_0=0}^h \sum_{k=0}^{s_0} (-1)^m z^{s_0} 
     \gkpSI{h}{s_0-k}_c \times [z^{k-2m}] S_{h,m,s}(z). 
\end{align*} 
Hence we arrive at the coefficient formula in (ii) by removing the sum 
indexed by $s_0$ in the last equation and setting $s_0 \mapsto n$ when 
$0 \leq n \leq h$. 
\end{proof} 

\subsection{Related expansions of the convergent numerator functions} 

\subsubsection*{Roadmap for the proof of the theorem} 

To prove our main result stated in Theorem \ref{theorem_MainGenabqz_J-FracThm}, 
it suffices to provide an argument inductively showing that for each 
$0 \leq n < h$ we have that 
\begin{align} 
\label{eqn_DivFnGF_finite_diff_eqn_stmt_v1} 
\frac{1-q}{1-q^{n+1}} & = [z^n] P_h(z) - \sum_{i=1}^{\min(n, h)} 
     [z^i] Q_h(z) \cdot \frac{1-q}{1-q^{n+1-i}} \\ 
\notag 
     & = [z^n] P_h(z) - \sum_{i=1}^{\min(n, h)} 
     [z^{n+1-i}] Q_h(z) \cdot \frac{1-q}{1-q^{i}}. 
\end{align} 
In this case we are concerned not only with the expansions of the $h^{th}$ 
convergent functions defined in terms of the denominator functions by the 
theorem, but also with the polynomial 
expansions of the numerator functions, $P_h(z)$, in $z$. 
By the definitions of the two recurrence relations in 
\eqref{eqn_ConvFn_PhzQhz_rdefs}, which are identical except for 
their initial conditions, we readily see by a simple two line proof that 
\begin{align} 
\label{eqn_Phz_from_Qhz_seq_subst_def}
P_h(z) & = Q_{h-1}(z) \quad \text{ \it substituting } \quad 
     c_i \longmapsto c_{i+1},\ \ab_i \longmapsto \ab_{i+1}. 
\end{align} 
Surprisingly, despite how seemingly closely related the expansions of these two 
convergent function subsequences are, the knowledge of 
Lemma \ref{lemma_ProdGFs_for_GenStirqCoeffs} and 
Lemma \ref{lemma_ExactExpsOfThe_ConvDenomFns} alone 
does not provide us with immediate working relations that allow us to 
combine like terms to simplify the right-hand-sides in 
\eqref{eqn_DivFnGF_finite_diff_eqn_stmt_v1}. 
That is to say that we need to be slightly more inventive to 
arrive at the comparitively simple inductive proof of 
Theorem \ref{theorem_MainGenabqz_J-FracThm} given in the next section. 

\subsubsection*{Statements of properties of the convergent numerator functions} 

\begin{definition}[Numerator Function Stirling $q$-Coefficients and Special Sums] 
In analogy to the product-based expansions of the generalized Stirling numbers we 
proved in Lemma \ref{lemma_ProdGFs_for_GenStirqCoeffs} of the last subsection, 
we define the following equivalent modified forms of these coefficients to 
explore the properties of $P_h(z)$ in more structured detail: 
\begin{align*} 
\gkpSI{h}{k}_{c,P} & = \gkpSI{h-1}{k}_{c,P} - c_{h+1} \gkpSI{h-1}{k-1}_{c,P} \\ 
     & = [z^k] (1-c_2z) (1-c_3z) \cdots (1-c_{h} z). 
\end{align*} 
We also define the next shifted forms of the 
special nested sums corresponding to the expansions of the 
numerator functions, $P_h(z)$, proved in the lemma below for 
\begin{align*} 
S_{h,m,s}^{[P]}(z) & = \sum_{i_1=2}^{h-2m} \sum_{i_2=i_1+2}^{h-2(m-1)} \cdots 
     \sum_{k_m=k_{m-1}+2}^h \left(\prod_{p=1}^m 
     \frac{ab_{k_p+1}}{(1-c_{k_p}z) (1-c_{k_p+1}z)} \right) 
     \Iverson{k_1+\cdots+k_m=s-m}. 
\end{align*} 
\end{definition} 

We state the next claim and prove the next lemma each providing 
apparent first relations between the two 
sets of generalized Stirling $q$-coefficients and between the two distinct 
convergent function sequences. 
These results provide a more complicated approach to proving the result 
outlined in \eqref{eqn_DivFnGF_finite_diff_eqn_stmt_v1} 
above by attempting to relate the numerator and 
denominator convergent functions through the characteristic expansions of the 
denominator functions which we use to expand $\Conv_h(a, b; q, z)$ 
in the theorem statement. 

\begin{claim}
For all integers $1 \leq k \leq h$ we have the following relations between the two 
distinct index-shifted sets of generalized Stirling $q$-coefficients: 
\begin{align*} 
\gkpSI{h}{k}_{c,P} & = \gkpSI{h-1}{k}_c + (c_1-c_h) \times 
     \sum_{1 \leq i_1 < \cdots < i_{k-1} < h} c_{i_1+1} \cdots c_{i_{k-1}+1} \\ 
     & = \gkpSI{h-1}{k}_c + (c_1-c_h) \times 
     [z^{k-1}] (1-c_2z) \cdots (1-c_{h-1}z). 
\end{align*} 
Moreover, for non-negative integers $h \geq 0$, $m \leq h/2$, and $s \leq h$, 
we have the next conjecture for a formula relating the differences of the two 
variants of the special nested sums we have defined above. 
\begin{align*} 
S_{h-1,m,s}(z) - S_{h,m,s}^{[P]}(z) & = 
     \sum_{\substack{2 \leq i_1 < \cdots < i_m \leq h \\ i_1 + \cdots + i_m = s}} 
     \left(\prod_{k=1}^m \frac{\ab_{i_k}}{(1-c_{i_k-1}z) (1-c_{i_k}z)} \right). 
\end{align*} 
\end{claim} 

\begin{lemma}[Exact Expansions of the Convergent Numerator Functions] 
For all $h \geq 2$ and $0 \leq n < h$, we have the next exact expansions of the 
convergent numerator functions, $P_h(z)$. 
\begin{align*} 
\tag{i} 
P_h(z) & = (1-c_2z)\cdots(1-c_hz)\left[1 + \sum_{m=1}^{\lfloor h/2 \rfloor} 
     \sum_{s=0}^{m(h+2)-2} \left(-z^2\right)^m S_{h-1,m,s}^{[P]}(z) \right] \\ 
\tag{ii} 
[z^n] P_h(z) & = \gkpSI{h}{n}_{c,P} + \sum_{m=1}^{\lfloor h/2 \rfloor} 
     \sum_{s=0}^{m(h+2)-2} \sum_{k=0}^h (-1)^m \gkpSI{h}{n-k}_{c,P} 
     [z^{k-2m}] S_{h-1,m,s}^{[P]}(z) 
\end{align*}
\end{lemma} 
\begin{proof}[Proof of (i) and (ii)]  
We cite a short proof that follows from the proofs of 
Lemma \ref{lemma_ExactExpsOfThe_ConvDenomFns} in the 
previous subsection by noticing that we have shown that 
\eqref{eqn_Phz_from_Qhz_seq_subst_def} results from the recurrence relations 
defining the two convergent function sequences in 
\eqref{eqn_ConvFn_PhzQhz_rdefs}. Thus we have already given an identical proof 
of these two related result in the previous subsection, and so we are done. 
\end{proof} 

\subsubsection*{The main expansion result for the numerator convergent functions} 

The next proposition provides a result that is relatively uncomplicated to 
prove by a nested, or \emph{double induction} procedure. In the next section we 
then combine this result with the $h$-order finite difference equations phrased in 
\eqref{eqn_DivFnGF_finite_diff_eqn_stmt_v1} implied by the 
rationality of $\Conv_h(q, z)$ for all $h \geq 1$ to complete our proof of 
Theorem \ref{theorem_MainGenabqz_J-FracThm}. 

\begin{prop}[Polynomial Expansions in $z$] 
\label{prop_Phqz_poly_exp_in_z_lemma} 
For each $h \geq 0$ and all $0 \leq n < h$, we have the following expansions 
of the polynomial coefficients of $P_h(q, z)$ in $z$: 
\begin{align*} 
[z^n] P_h(q, z) & = \sum_{i=0}^n [z^i] Q_h(q, z) \cdot \frac{1-q}{1-q^{n+1-i}}. 
\end{align*} 
\end{prop} 
\begin{proof} 
We proceed to prove this result by a double induction procedure on $h$ with 
$n \in [0, h)$ for each fixed $h \geq 1$. Let the shorthand notation for the 
coefficients of the numerator functions with respect to $z$ be defined as 
$p_{h,n} := [z^n] P_h(z)$. 
We first show that this result is true in the first cases where 
$h := 1, 2$. 
More precisely, for $h := 1, 2$ we have the next expansions of the functions, 
$P_h(z)$, given by 
\begin{align*}
P_1(q, z) & = 1 \\ 
P_2(q, z) & = 1 - \frac{2q(1-q)z}{(1-q^4)}, 
\end{align*} 
which by computation agrees with the right-hand-side formula stated above. 
Next, we assume that the formula is correct for some $h \geq 1$ and all 
$0 \leq n < h$. We then proceed to 
expand the finite sum formula above using the recurrences in 
\eqref{eqn_ConvFn_PhzQhz_rdefs} as 
\begin{align*} 
\sum_{i=0}^n [z^i] Q_h(q, z) \cdot \frac{1-q}{1-q^{n+1-i}} & = 
     \sum_{i=0}^n \Bigl[[z^i] Q_{h-1}(q, z) -c_h(q) [z^{i-1}] Q_{h-1}(q, z) \\ 
     & \phantom{=\sum_{i=0}^n [z^i]\Bigl[Q_{h-1}(q, z)\ } - 
     \ab_h(q) [z^{i-2}] Q_{h-2}(q, z)\Bigr] \cdot \frac{1-q}{1-q^{n+1-i}} \\ 
     & = 
     \sum_{i=0}^n [z^i] Q_{h-1}(q, z) \cdot \frac{1-q}{1-q^{n+1-i}} - 
     c_h(q) p_{h-1,n-1} - \ab_h(q) p_{h-2,n-2}, 
\end{align*} 
which is equal to $[z^n] P_h(q, z)$ if the leftmost sum term is equal to 
$p_{h-1,n}$. We know that this occurs when $n < h-1$, but we also see that when 
$n = h-1$ that $p_{h-1,n} = 0$, so we may assume that this summation term 
equals $p_{h-1,n}$ for all $0 \leq n < h$. 
\end{proof} 

\section{Proof of the main theorem and restatements of these results} 
\label{Section_ProofsOfMainResults} 

\subsection{Proof of the theorem and convergence as $h \longrightarrow \infty$} 

\begin{proof}[Proof of Theorem \ref{theorem_MainGenabqz_J-FracThm} (Sketch)] 
The expansion in \eqref{eqn_mainthm_stmt_part1} follows 
from well-known recursive properties of the 
convergents to any J-fractions of the form in \eqref{eqn_J-Fraction_Expansions} 
\citep[\S 1.12]{NISTHB}. 
By the rationality of $\Conv_h(q, z)$ for each $h \geq 1$, we have the following 
$h$-order finite difference equation satisfied exactly by the 
sequence of coefficients 
enumerated by $\Conv_h(q, z)$ for all integers $n \geq 0$: 
\begin{align*} 
[z^n] \Conv_h(q, z) & = [z^n] P_h(q, z) - \sum_{i=1}^{\min(n, h)} 
     [z^i] Q_h(q, z) \cdot [z^{n-i}] \Conv_h(q, z). 
\end{align*} 
For each fixed $h \geq 2$, we prove that the first form of 
\eqref{eqn_DivFnGF_finite_diff_eqn_stmt_v1} restated as 
\begin{align} 
\tag{i} 
\frac{1-q}{1-q^{k+1}} & = [z^k] P_h(z) \Iverson{0 \leq k < h} - 
     \sum_{i=1}^{\min(k, h)} 
     [z^i] Q_h(z) \cdot \underset{[z^{k-i}] \Conv_h(q, z)}{\underbrace{ 
     \frac{1-q}{1-q^{k+1-i}}}}, 
\end{align} 
holds for all $k \in [0, h-1]$ by induction on $n$. 
Since $\Conv_h(q, z)$ is rational for all $h \geq 1$, this result 
suffices to prove \eqref{eqn_Convh_coeffs_zn_eq_ratio}. 
Then since we have the known property that $[z^n] \Conv_h(z) = [z^n] J_{\infty}(z)$ 
for all $0 \leq n < h$ (more precisely, for all $0 \leq n < 2h$), this result 
combined with \eqref{eqn_mainthm_stmt_part1} suffices to prove that 
\eqref{eqn_mainthm_stmt_part3} holds in the limiting case 
as $h \longrightarrow \infty$ provided the infinite continued fraction converges. 
We prove the convergence of the infinite J-fraction expansion defined by our 
sequences in Definition \ref{def_ParamSeqDefs_and_Notations} 
using \emph{Pringsheim's theorem} in the proposition stated immediately below. \\ 
\end{proof} 
\begin{proof}[Proof of \eqref{eqn_Convh_coeffs_zn_eq_ratio} by Induction on $n$.] 
Suppose that $h \geq 2$ is fixed. Since $[z^0] P_h(q, z) = 1$ for all $h$, 
it follows that (i) is true when $k = 0$. 
We next suppose that (i) is true for all $k < n$ and proceed to show that this 
implies that the statement is true for $k = n$. 
We first notice that the indices $n-i$ when 
$i \in [1, n]$ are in the range of $\{0, 1, \ldots, n-1\}$, so that we see that 
the right-hand-side of (i) is equal to $[z^n] \Conv_h(q, z)$ by our inductive 
hypothesis. Then by Proposition \ref{prop_Phqz_poly_exp_in_z_lemma}, we see 
that (i) is true for this choice of $k = n$. Hence our claim is true for all 
$n \in [0, h-1]$ when $h \geq 2$ is a fixed positive integer. 
\end{proof} 

\begin{prop}[Convergence of the Infinite J-Fraction Expansion, $J_{\infty}(q, z)$] 
\label{prop_ConvOfInfJ-Fraction_ConvOfThm} 
For fixed non-zero $q \in \mathbb{C}$ such that $0 < |q| < 0.206783$, 
we have that the infinite J-fraction 
\[
J_{\infty}(q, z) = (1-q) \times \sum_{n \geq 0} \frac{z^n}{1-q^{n+1}}, 
\] 
generating the divisor function is convergent with 
$\Conv_h(q, z) \longrightarrow J_{\infty}(q, z)$ in the limiting case as 
$h \longrightarrow \infty$. Moreover, the infinite sum for 
$J_{\infty}(q, z)$ expanded in \eqref{eqn_mainthm_stmt_part3} 
converges uniformly as a function of $z$. In particular, we may differentiate this 
sum representing the infinite J-fraction termwise with respect to $z$. 
\end{prop} 
\begin{proof}[Proof of Convergence] 
We define the sequences $a_h, b_h$ as in the following equations and proceed to 
use \emph{Pringsheim's theorem} to prove convergence of 
$\Conv_{\infty}(q, z) \longrightarrow J_{\infty}(z)$ when $|q| < 0.206783$ 
\citep[\S 1.12(v)]{NISTHB}: 
\begin{align*} 
a_h(q) & := \frac{z^2 q^{2i-3} \left(1-q^{i-1}\right)^4}{(1-q^{2i-3}) (1-q^{2i-2})^2 
     (1-q^{2i-1})} \\ 
b_h(q) & : = \frac{1-z q^{i-2}\left(2q + q^{2i} - q^i - q^{i+1} -q^2-q^3+q^{i+2} 
     \right)}{(1-q^{2i-2})(1-q^{2i})}. 
\end{align*} 
In particular, if we can show that $|b_h(q)| \geq |a_h(q)| + 1$ for all sufficiently 
large $h \geq h_0$, then we have that 
$\Conv_h(q, z) \longrightarrow J_{\infty}(q, z)$ as $h \longrightarrow \infty$,  
that $|\Conv_h(q, z)| < 1$ for all $h$, and that our infinite J-fraction satisfies 
$|J_{\infty}(q, z)| < 1$. 
Since we will be setting $z \mapsto q$ in the generating functions for the 
divisor and generalized sum-of-divisors functions in the 
results of the next subsection, 
we may assume that $z \equiv q$, and in particular that $|z| = |q|$. 

We first expand several inequalities for $|a_i| + 1$ when $i \geq h_0$ 
is taken to be a sufficiently large positive integer: 
\begin{align*} 
|a_i| + 1 & = \frac{|q|^{2i-1} |1-q^{i-1}|^2}{|1-q^{2i-3}| |1-q^{2(i-1)}| 
     |1-q^{2i-1}|} + 1 \\ 
     & \leq 
     \frac{|q|^{2i-1} \left(1+|q|^{i-1}\right)^2}{\left(1-|q|^{2i-3}\right) 
     \left(1-|q|^{2(i-1)}\right) \left(1-|q|^{2i-1}\right)} + 1 \\ 
     & \leq 
     \frac{|q|^{2i-1} \left(1+|q|^{i-1}\right)^2 + (1-|q|^{i-1})^4}{ 
     \left(1-|q|^{i-1}\right)^4}. 
\end{align*}
Secondly, we expand a few corresponding inequalities for $|b_i|$ in the 
following forms: 
\begin{align*} 
|b_i| & = 
     \frac{|1+q^{i-1}\left(2q+q^{2i}+q^{i+2}\right) - 
     q^{i-1}\left(q^i q^{i+1} + q^2 + q^3\right)|}{|1-q^{2(i-1)}| |1-q^{2i}|} \\ 
     & \geq 
     \frac{|1+q^{i-1}\left(2q+q^{2i}+q^{i+2}\right) - 
     q^{i-1}\left(q^i q^{i+1} + q^2 + q^3\right)|}{ 
     \left(1+q^{2(i-1)}\right) \left(1+q^{2i}\right)} \\ 
     & \geq 
     \frac{|1+q^{i-1}\left(2q+q^{2i}+q^{i+2}\right) - 
     q^{i-1}\left(q^i q^{i+1} + q^2 + q^3\right)|}{ 
     \left(1+q^{i-1}\right)^2} \\ 
     & \geq 
     \frac{C_{q,i}^2 \times |1+q^{i-1}\left(2q+q^{2i}+q^{i+2}\right) - 
     q^{i-1}\left(q^i q^{i+1} + q^2 + q^3\right)|}{ 
     \left(1-q^{i-1}\right)^2},\ \\ 
     & \phantom{\geq\ } 
     \text{ for some } 
     0 < C_{q,i} \leq \frac{(1-|q|^{i-1})^2}{1+|q|^{2i-2}} \leq 1. 
\end{align*} 
Thus if we let $t := q^{i-1}$, it follows that if we can show that there is a 
positive constant, $C_{q,i}$, depending on $q$ and the $i \geq h_0 \geq 1$ 
satisfying the above inequality and the condition that 
\begin{align*} 
\tag{i} 
C_{q,i}^2 \left|1 + q(q^2+q-2) t + q(1-q) t^2 - q^2 t^3\right| & \geq 
     \left(1-|t|\right)^4 + |t|^2 (1+|t|)^2, 
\end{align*} 
then we have proven that $|b_i| \geq |a_i| + 1$ for all sufficiently large 
$i \geq 1$. The condition in (i) combined with the inequality defining $C_{q,i}$ 
in the previous equations for $|b_i|$ and a tentative technical requirement that 
$|q^2+q-2|, |1-q| \leq 1$, imply that we must find such a $C_{q,i}$ satisfying 
\begin{align*} 
\sqrt{\frac{\left(1-|t|\right)^4 + |t|^2 (1+|t|)^2}{1 + |t| + |t|^2 + |t|^3}} 
     & \leq 
     \sqrt{\frac{\left(1-|t|\right)^4 + |t|^2 (1+|t|)^2}{ 
     1 + |q(q^2+q-2)| t + |q(1-q)| |t|^2 - |q|^2 |t|^3}} \\ 
     & \leq 
     \sqrt{\frac{\left(1-|t|\right)^4 + |t|^2 (1+|t|)^2}{ 
     \left|1 + q(q^2+q-2) t + q(1-q) t^2 - q^2 t^3\right|}} \\ 
     & \leq C_{q,i} \leq \frac{(1-|t|)^2}{1+|t|^2} \leq 1, 
\end{align*} 
By numerical computations the requisite inequality 
\begin{align*} 
\frac{(1-|t|)^2}{1+|t|^2} & \geq 
     \sqrt{\frac{\left(1-|t|\right)^4 + |t|^2 (1+|t|)^2}{1 + |t| + |t|^2 + |t|^3}}, 
\end{align*} 
is satisfied for $|t|$ in the approximate range $0 < |t| < 0.206783$. 
Then we finally see that since $0 < |q|^{i+1} < |q|^i$ for all $i \geq 1$, 
if $0 < |q| < 0.206783$ we have convergence of our infinite J-fraction. 
\end{proof}
\begin{proof}[Proof of the Uniform Convergence of \eqref{eqn_mainthm_stmt_part3}] 
Moreover, we can prove easily that the infinite sum for $J_{\infty}(q, z)$ 
defined in \eqref{eqn_mainthm_stmt_part3} is uniformly convergent as a function of 
$z$. For the parameter setting of $(a, b) := (q, q^2)$ in the infinite J-fraction, 
$J_{\infty}(a, b; q, z)$, let the inner terms in this sum be denoted by 
\begin{align*} 
\Conv_{i+1,q,z} & := \frac{q(1-q) (-1)^{i-1} q^{i^2} (q; q)_i^4 z^{2i}}{ 
     (q; q^2)_i^2 (q^2; q^2)_i \times Q_{i}(q, z) Q_{i+1}(q, z)}. 
\end{align*} 
Then for some positive constant $A$, a fixed function $b_q$, and non-zero 
$q, z \in \mathbb{C}$ such that $|c_q z|, |qz| < 1$ where $|q| < 0.206783$ 
as above, we can bound these inner sum terms by 
\begin{equation*} 
|\Conv_{i,q,z}| \leq A \cdot b_q^i q^{i^2} z^{2i} < A, 
\end{equation*} 
in which case the \emph{Weierstrass M-test} implies our result. That is, the 
convergence of the sum in \eqref{eqn_mainthm_stmt_part3} is uniform in $z$. 
\end{proof} 

\begin{cor}[A Complete Generating Function for the Divisor Function] 
\label{cor_GF_for_the_DivisorFn_dn} 
Let the generating function, $D_{0,h}(q, z)$ be defined as 
\begin{align} 
\label{eqn_cor_GF_for_the_DivisorFn_dn_D0hqzFn_def} 
D_{0,h}(q, z) & := \frac{q(1+q)}{(1-q)(1+q-z)} + \sum_{j=1}^{h-1} 
     \frac{q \cdot q^{j^2} (q; q)_j^4 \times z^{2j}}{(q; q^2)_j^2 (q^2; q^2)_j^2} 
     \times \widetilde{D}_{0,j}(q, z)^{-1}, 
\end{align} 
where the function, $\widetilde{D}_{0,j}(q, z)$, is defined by the 
$q$-coefficient expansions 
\begin{align} 
\label{eqn_Qim1Qi_ConvFnSum_Denom_Exp} 
\widetilde{D}_{0,j}(q, z) & := 
     \sum_{n=0}^{2j} \gkpSI{j+1}{n}_{q,q,q^2} \gkpSI{j}{2j-n} z^n \\ 
\notag 
     & \phantom{=} + 
     \sum_{n=0}^{2j+1} \Biggl( 
     \sum_{\substack{1 \leq m_1 \leq \left\lfloor \frac{j}{2} \right\rfloor \\ 
                     1 \leq m_2 \leq \left\lfloor \frac{j+1}{2} \right\rfloor}} 
     \sum_{\substack{1 \leq s_1 \leq mj \\ 1 \leq s_2 \leq m(j+1)}} 
     \sum_{\substack{1 \leq k_1 \leq s_1 \\ 1 \leq k_2 \leq s_2}} 
          \gkpSI{j+1}{n-k_2}_{q,q,q^2} \gkpSI{j}{2j+1-n-k_1}_{q,q,q^2} 
          \times \\ 
\notag 
     & \phantom{:=\sum\Biggl(\sum\sum\sum\ } \times 
     (-1)^{m_1+m_2} [z^{k_1-2m_1}] S_{j,m_1,s_1}(z) \cdot 
     [z^{k_2-2m_2}] S_{j+1,m_2,s_2}(z) 
     \Biggr) z^n \\ 
\notag 
     & \phantom{=} + 
     \sum_{n=0}^{2j+1} \sum_{m=1}^{\lfloor (j+1)/2 \rfloor} \sum_{s=0}^{m(j+1)} 
     \sum_{k=0}^s \gkpSI{j+1}{n-k}_{q,q,q^2} \gkpSI{j}{2j+1-n}_{q,q,q^2} 
     (-1)^m [z^{k-2m}] S_{j+1,m,s}(z) \times z^n \\ 
\notag 
     & \phantom{=} + 
     \sum_{n=0}^{2j+1} \sum_{m=1}^{\lfloor j/2 \rfloor} \sum_{s=0}^{mj} 
     \sum_{k=0}^s \gkpSI{j}{n-k}_{q,q,q^2} \gkpSI{j+1}{2j+1-n}_{q,q,q^2} 
     (-1)^m [z^{k-2m}] S_{j,m,s}(z) \times z^n. 
\end{align} 
Then for all $h \geq 2$ and $0 \leq n < h$, we have that 
$d(n) = [q^n] D_{0,h}(q, q)$ whenever $0 < |q| < 0.206783$ and 
that the limiting case yields the convergent sum 
\begin{align*} 
D_{0,\infty}(q, q) & = \sum_{n \geq 1} \frac{q^n}{1-q^n} = \sum_{m \geq 1} d(m) q^m. 
\end{align*} 
\end{cor} 
\begin{proof} 
These two results are only a restatement of 
Theorem \ref{theorem_MainGenabqz_J-FracThm}, 
whose convergence is guaranteed by 
Proposition \ref{prop_ConvOfInfJ-Fraction_ConvOfThm}, 
in light of the expansions of the convergent denominator functions, $Q_h(q, z)$, 
given by Lemma \ref{lemma_ExactExpsOfThe_ConvDenomFns}. 
\end{proof} 

\subsection{Modified generating functions for the integer-order 
            sums-of-divisors functions} 
\label{subSection_ModifiedGFs_for_IntOrderDivFns} 

This section employs the uniform convergence of the right-hand-side sum in 
\eqref{eqn_mainthm_stmt_part3} from Theorem \ref{theorem_MainGenabqz_J-FracThm} 
to formulate new generating functions for the 
generalized sums-of-divisors functions, 
$\sigma_{\alpha}(n)$ when $\alpha \geq 1$ is integer-valued. 
In particular, we borrow a more general result from the reference 
\citep[\S 2]{SQSERIESMDS} which states that given any $m \in \mathbb{Z}^{+}$ and 
any sequence, $\langle f_n \rangle$, whose ordinary generating function (OGF), 
$F_f(z)$, has higher-order derivatives of orders $j$ for all $0 \leq j \leq m$, 
we have a transformation of this OGF into the OGF enumerating the modified sequence of 
$\langle n^m f_n \rangle$ given by the following finite sum\footnote{ 
     The \emph{Stirling numbers of the second kind}, $\gkpSII{n}{k}$, are also 
     commonly denoted by $S(n, k)$ for integers $n, k \geq 0$ such that 
     $0 \leq k \leq n$ \citep[\S 26.8]{NISTHB}. 
}: 
\begin{align} 
\label{eqn_nPowmTimesSeqfn_OGF_transform_stmt_v1} 
\sum_{n \geq 0} n^m f_n z^n & = \sum_{j=0}^m \gkpSII{m}{j} z^j F_f^{(j)}(z). 
\end{align} 

\begin{prop}[Generating Functions for the Generalized sums-of-divisors Functions] 
\label{prop_GFs_for_SigmaAlphan_AlphaAPosInt} 
Let the notation for the functions, $\widetilde{D}_{0,j}(q, z)$, be defined 
as in the statement of Corollary \ref{cor_GF_for_the_DivisorFn_dn} 
in the previous subsection. 
Then for each fixed integer $\alpha \geq 1$, we have the generating functions 
for the sums-of-divisors functions, $\sigma_{\alpha}(n)$, expanded in terms of the 
next finite sums involving the Stirling numbers of the second kind. 
\begin{align*} 
D_{\alpha,\infty}(q) & = \sum_{m \geq 1} \sigma_{\alpha}(m) q^m = 
     \sum_{n \geq 1} \frac{n^{\alpha} q^n}{1-q^n} \\ 
     & = 
     \sum_{i=0}^{\alpha} \gkpSII{\alpha}{i} \left(
     \frac{z^i i! q(1+q)}{(1-q)(1+q-z)^{i+1}} + 
     \sum_{j \geq 1} \frac{q \cdot q^{j^2} (q; q)_j^4 \times z^{i}}{ 
     (q; q^2)_j^2 (q^2; q^2)_j^2} \times 
     D^{(i)}\left[\frac{z^{2j}}{\widetilde{D}_{0,j}(q, z)}\right]
     \right) \Bigg|_{z=q} 
\end{align*} 
\end{prop} 
\begin{proof} 
This result is an immediate consequence of our new transformation result in 
\eqref{eqn_nPowmTimesSeqfn_OGF_transform_stmt_v1} applied to the statement of 
Corollary \ref{cor_GF_for_the_DivisorFn_dn}, where we know by 
Proposition \ref{prop_ConvOfInfJ-Fraction_ConvOfThm} that 
we may differentiate the sum in \eqref{eqn_cor_GF_for_the_DivisorFn_dn_D0hqzFn_def} 
termwise with respect to $z$. 
\end{proof} 

\begin{cor}[Special Cases] 
For the special cases of $\alpha := 1, 2$, we have more explicit expansions of the 
generating functions for the sums-of-divisors functions, 
$\sigma_1(n)$ and $\sigma_2(n)$, given by 
\begin{align*} 
D_{1,\infty}(q) & = \sum_{m \geq 1} \sigma_{1}(m) q^m = 
     \sum_{n \geq 1} \frac{n^{} q^n}{1-q^n} \\ 
     & = 
     \frac{q^2(1+q)}{1-q} + 
     \sum_{j \geq 1} \frac{q \cdot q^{j^2} (q; q)_j^4}{ 
     (q; q^2)_j^2 (q^2; q^2)_j^2} \times \left(
     \frac{2j q^{2j}}{\widetilde{D}_{0,j}(q, q)} - 
     \frac{q^{2j+1}\widetilde{D}_{0,j}^{\prime}(q, q)}{\widetilde{D}_{0,j}(q, q)^2} 
     \right) \\ 
D_{2,\infty}(q) & = \sum_{m \geq 1} \sigma_{2}(m) q^m = 
     \sum_{n \geq 1} \frac{n^{2} q^n}{1-q^n} \\ 
     & = 
     \frac{q^2(1+q)(1+2q)}{1-q} + 
     \sum_{j \geq 1} \frac{q \cdot q^{j^2} (q; q)_j^4}{ 
     (q; q^2)_j^2 (q^2; q^2)_j^2} \times \Biggl(
     \frac{4j^2 q^{2j}}{\widetilde{D}_{0,j}(q, q)} - 
     \frac{(4j+1) q^{2j+1}\widetilde{D}_{0,j}^{\prime}(q, q)}{ 
     \widetilde{D}_{0,j}(q, q)^2} \\ 
     & \phantom{=\frac{q^2(1+q)(1+2q)}{1-q} +\sum\ } - 
     \frac{q^{2j+1} \left( 
     \widetilde{D}_{0,j}(q, q)\widetilde{D}_{0,j}^{\prime\prime}(q, q) - 
     2 \widetilde{D}_{0,j}^{\prime}(q, q)^2\right)}{\widetilde{D}_{0,j}(q, q)^3} 
     \Biggr), 
\end{align*} 
where the derivatives of $\widetilde{D}_{0,j}(q, z)$ are partial derivatives taken 
with respect to the second parameter $z$ whose forms are easily expanded by 
differentiating the polynomial functions of $z$ in 
\eqref{eqn_Qim1Qi_ConvFnSum_Denom_Exp}. 
\end{cor} 
\begin{proof} 
These results follow from Proposition \ref{prop_GFs_for_SigmaAlphan_AlphaAPosInt} 
combined with two computations with the quotient rule providing that for any 
non-zero function, $G(z)$, and integers $j \geq 1$, the 
first and second derivatives of the quotient, 
$z^{2j} / G(z)$, are expanded by 
\begin{align*} 
z \cdot D\left[z^{2j} / G(z)\right] & = \frac{2j z^{2j+1}}{G(z)} - 
     \frac{z^{2j} G^{\prime}(z)}{G(z)^2} \\ 
z^2 \cdot D^2\left[z^{2j} / G(z)\right] & = \frac{2j(2j-1)z^{2j}}{G(z)} - 
     \frac{4jz^{2j+1}G^{\prime}(z)}{G(z)^2} - 
     \frac{z^{2j+2}\left(G(z)G^{\prime\prime}(z)-2G^{\prime}(z)^2\right)}{G(z)^3}. 
     \qedhere 
\end{align*} 
\end{proof} 

\section{Significance of the new results, applications, and 
         relations to recent research} 
\label{Section_Sig_and_Apps} 

\subsection{Constructions of generating functions enumerating other 
            special $q$-series} 

Our method of proof in Section \ref{Section_ProofsOfMainResults} hints at, and 
in the most general cases applies to, the 
J-fraction expansions of the most general forms of 
\eqref{eqn_J-Fraction_Expansions} for arbitrary sequences, 
$\langle \ab_i \rangle$ and $\langle c_i \rangle$. In fact, when our so-termed 
notions of the generalized Stirling $q$-coefficients in 
\eqref{eqn_Qim1Qi_ConvFnSum_Denom_Exp} of 
Corollary \ref{cor_GF_for_the_DivisorFn_dn} are replaced with the coefficients, 
$\gkpSI{h}{k}_c$, defined as in \eqref{eqn_GenStirqCoeffs_GenS1hkc} of 
Definition \ref{def_SpSums_GenStirlingqCoeffs}, then 
for any \emph{convergent} J-fraction expansion of 
\eqref{eqn_J-Fraction_Expansions} we have new sums for the $h^{th}$ convergents to 
$J_{\infty}(z)$ expanded more generally by the corresponding formula in 
\eqref{eqn_mainthm_stmt_part1} of Theorem \ref{theorem_MainGenabqz_J-FracThm}. 
Thus new infinite sums for other known J-fraction expansions, such as those 
considered in the references and in Table \ref{table_OtherqSeriesJFracExps} below, 
are formulated by an appeal to the generalized 
results for these continued fractions we have established within this article. 

\begin{table}[ht!] 
\centering\small

\begin{tabular}{||c||c|c||c||} \hline\hline
$[z^n] J_{\infty}(z)$ & $c_1$ & $c_h$ for $h \geq 2$ & $\ab_h$ \\ 
\hline\hline
$\QPochhammer{a}{q}{n}$ & $1-a$ & $q^{h-1} - a q^{h-2} \left(q^{h} + q^{h-1} - 1\right)$ & 
                          $a q^{2h-4} (a q^{h-2}-1)(q^{h-1}-1)$ \\ 
$\frac{1}{\QPochhammer{q}{q}{n}}$ & $\frac{1}{1-q}$ & 
     $\frac{q^{h-1} \left(q^{h-1} \QBinomial{h-1}{1}{q} - 
      \QBinomial{h-2}{1}{q}\right)}{ 
      \QBinomial{2h-3}{1}{q} (q^{2h-1}-1)}$ & 
     $-\frac{q^{3h-5}}{(q^{2h-3}-1)^2 (1 + q^{h-2} + q^{h-1} + q^{2h-3})}$ \\ 
$\frac{q^{\binom{n}{2}}}{\QPochhammer{q}{q}{n}}$ & $\frac{1}{1-q}$ & 
     $\frac{q^{h-2} (1-q) \left(q^h \gkpSI{h-2}{1}_q - \gkpSI{h-1}{1}_q\right)}{ 
     (1-q^{2h-3})(1-q^{2h-1})}$ & 
     $\begin{cases} 
      -\frac{1}{(1-q)^2 (1+q)}, & \text{ if $h = 2$; } \\ 
      -\frac{q^{\binom{h}{2}}}{(1-q^h)^2 (1+q^{h-2}(1+q) + q^{2h-3})}, & 
      \text{ if $h \geq 3$ } 
      \end{cases}$ \\ 
$\QPochhammer{z q^{-n}}{q}{n}$ & $\frac{q-z}{q}$ & 
     $\frac{q^h - z - qz + q^h z}{q^{2h-1}}$ & 
     $\frac{(q^{h-1}-1) (q^{h-1}-z) \cdot z}{q^{4h-5}}$ \\ 
$\frac{1}{\QPochhammer{z q^{-n}}{q}{n}}$ & $\frac{q}{q-z}$ & 
     $\frac{q^{h-1} \left(q^{2h-2} + z + q^{h-1} z - q^{h} z\right)}{ 
      (q^{2h-3}-z) (q^{2h-1}-z)}$ & 
     $\QBinomial{h-1}{1}{q} \cdot \frac{q^{3h-4}(1-q)(q^{h-2}-z) \cdot z}{ 
      (q^{2h-4}-z) (q^{2h-3}-z)^2 (q^{2h-2}-z)}$ \\ 
$\frac{(a; q)_n}{(b; q)_n}$ & $\frac{1-a}{1-b}$ & $\frac{q^{i-2}\left(q+ab q^{2i-3} + a\left(1-q^{i-1}-q^i\right) 
     + b \left(-1-q+q^{i}\right)\right)}{ 
     (1-b q^{2i-4}) (1-b q^{2i-2})}$ & $\frac{q^{2i-4} (1-bq^{i-3}) (1-a q^{i-2}) (a - b q^{i-2}) 
     (1-q^{i-1})}{(1-b q^{2i-5}) (1-b q^{2i-4})^2 (1-b q^{2i-3})}$ \\ 
\hline\hline 
\end{tabular} 

\bigskip 
\caption{J-fraction parameters generating the terms in special $q$-series expansions} 
\label{table_OtherqSeriesJFracExps} 

\end{table} 

Examples of other notable $q$-series expansions derived from the generalized 
Jacobi-type J-fractions defined in Section \ref{subSection_JFracExps_of_OGFs} 
are summarized in Table \ref{table_OtherqSeriesJFracExps} on page 
\pageref{table_OtherqSeriesJFracExps} \cite[\cf \S 4]{SQSERIES-CFRACS}. 
We can use these continued fraction expansions combined with the definition of the 
more general Stirling $q$-coefficients in \eqref{eqn_GenStirqCoeffs_GenS1hkc} and 
Lemma \ref{lemma_ExactExpsOfThe_ConvDenomFns} to generate new $q$-series 
expansions for 
$q$-exponential functions, $q$-trigonometric functions, infinite $q$-Pochhammer 
symbol products and their reciprocals, and related series such as for the 
Rogers-Ramanujan continued fractions \citep{BERNDT-QSERIES} \citep[\S 17.3]{NISTHB}. 
Generalized forms of the Stirling numbers of the first kind in the context of the 
product-based definitions we used to define our notions of the 
Stirling $q$-coefficients in the previous section are considered in the references 
\citep{MANSOUR-GENSTIRNUMS,JNT-JACOBI-STIRLINGNUMBERS}. 
Other typical $q$-analogs to the Stirling numbers, or $q$-Stirling numbers, are 
defined in \citep{CHARALAM-GENQ-STIRLINGNUMS,FSTIRNUMS} and in the 
particular form of 
\citep[\S 17.3]{NISTHB} 
\begin{align*} 
a_{m,s}(q) & = \frac{(1-q)^s}{q^{\binom{s}{2}} (q; q)_s} \times \sum_{j=0}^s 
     \gkpSI{s}{j}_q \frac{(-1)^j q^{\binom{j}{2}} (1-q^{s-j})^m}{(1-q)^m}. 
\end{align*} 
Charalambides defines a $q$-analog to the Stirling numbers of the first kind by the 
products 
\begin{align*} 
([t]_q-[r]_q)([t]_q-[r+1]_q) \cdots ([t]_q-[r+n-1]_q) & = 
     \sum_{k=0}^n s_q(n, k; r) [t]_q^k, 
\end{align*} 
where $[r]_q := (1-q^x) / (1-q)$ denotes a \emph{$q$-real number}, 
$s_q(n, k; r)$ corresponds to a modified form of 
\eqref{eqn_GenStirqCoeffs_GenS1hkc} with $c_h := [n+r-1]_q$, 
and where we have an explicit finite sum expanded in terms of the 
$q$-binomial coefficients by the formula 
\begin{align*} 
s_q(n, k; r) & = \frac{1}{(1-q)^{n-k}} \sum_{j=k}^n \gkpSI{n}{j}_q \binom{j}{k} 
     (-1)^{j-k} q^{\binom{n-j}{2}+r(n-j)}. 
\end{align*} 
Additionally, we can extend the generating function results for the 
J-fraction expansions defined by Definition \ref{def_ParamSeqDefs_and_Notations} to 
obtain new identities and generating function expansions of series involving 
ratios of two $q$-Pochhammer symbols. For example, we have the following 
$q$-series identities which are either direct applications of the new 
$q$-Pochhammer ratio series or that form special cases of our new results 
corresponding to other special $q$-series expansions where 
$(a; q)_{-n} = \prod_{k=1}^n (1-a / q^k)^{-1} = (a/q; 1/q)_n^{-1}$ 
\citep{BERNDT-QSERIES}: 
\begin{align*} 
\sum_{n=-\infty}^{\infty} \frac{(a; q)_n}{(b; q)_n} z^n & = 
     \frac{(az; q)_{\infty} (q / (az); q)_{\infty} (q; q)_{\infty} (b/a; q)_{\infty}}{ 
     (z; q)_{\infty} (b / (az); q)_{\infty} (b; a)_{\infty} (q/a; q)_{\infty}},\ 
     \left| \frac{b}{a} \right| < |z| < 1, |q| < 1 \\ 
\sum_{n \geq 0} \frac{(a; q)_n}{(q; q)_n} z^n & = 
     \frac{(az; q)_{\infty}}{(z; q)_{\infty}},\ |z|, |q| < 1 \\ 
     & = 
     \sum_{n \geq 0} \frac{(a; q)_n (az; q)_n (qz)^n q^{n(n-1)} (1-az q^{2n})}{ 
     (q; q)_n (z; q)_{n+1}} \\ 
\sum_{n \geq 0} \frac{q^{n^2} z^n}{(q; q)_n} & = 
     \sum_{n \geq 0} \frac{(z; q^2)_n}{(zq; q^2)_n} z^n. 
\end{align*} 
Special cases of the last identity and the entries for the J-fraction expansions 
given in Table \ref{table_OtherqSeriesJFracExps} 
similarly allow us to generate the infinite $q$-Pochhammer symbol product, 
$(z; q)_{\infty}$, and its reciprocal, which provides immediate applications to 
generating functions and rational approximate generating functions 
for partition functions. 

\subsection{Relations to recent research and some open problems} 

Recent work on the divisor function and the sum-of-divisors function using 
results obtained from Lambert series identities and the Lambert series 
generating functions for the generalized sum-of-divisors functions defined by 
\begin{align*} 
L_{\alpha}(q) & := \sum_{n \geq 1} \frac{n^{\alpha} q^n}{1-q^n} = 
     \sum_{m \geq 1} \sigma_{\alpha}(m) q^m, 
\end{align*} 
are considered in the references 
\citep{JNT-COMBINTERPRET-DIVFNS,JNT-NEWCVLSDN,JNT-ANEWLOOK,SEGLADUN-GFS}. 
The applications of these results are related to divisor sum convolutions, 
combinatorial interpretations of the explicit values of the divisor function, 
$d(n) \equiv \sigma_0(n)$, and of course finding new generating functions for the 
two classical divisor functions, $d(n)$ and $\sigma(n)$, and the 
generalized sum-of-divisors functions, $\sigma_{\alpha}(n)$. 
Our new results proved within the article provide new $q$-series expansions of the 
generating functions for the divisor functions and the generalized 
sum-of-divisors functions whose $h$-order accurate $h^{th}$ convergents are 
rational for each finite $h \geq 2$ 
(\cf equation \eqref{footnote_ApproximateRationalGFExamples} 
on page \pageref{footnote_ApproximateRationalGFExamples} below). 
Recent results on continued fraction expansions similar to the J-fraction 
expansions defined by Definition \ref{def_ParamSeqDefs_and_Notations} and for 
other special $q$-series expansions are considered in 
\citep{QBC-CFRACS,FLOWERS-CFRACS-LSERIES}. 

If we set $q \mapsto 1^{-}$ in our expansions for the $q$-Pochhammer ratios from 
Definition \ref{def_ParamSeqDefs_and_Notations}, we obtain a special case of the 
generalized hypergeometric function, $_1F_1(a; b; z)$. 
Several new results proved in \citep{LAMBERT-SERIES-NEAR} 
similarly provide asymptotic formulas for the 
sum-of-divisors generating functions near $q = 1$. 
In particular, we may define the 
Lambert series, $\mathcal{L}_q(s, x)$, where $\mathcal{L}_q(s, 1)$ 
corresponds to the ordinary generating function for $\sigma_s(n)$, as 
\begin{align*} 
\mathcal{L}_q(s, x) & := \sum_{n \geq 1} \frac{n^s q^{nx}}{1-q^n} = 
     \frac{\Gamma(s+1) \zeta(s+1, x)}{\left(\log \frac{1}{q}\right)^{s+1}} - 
     \sum_{n \geq 0} \frac{\zeta(1-s-n) B_n(x) (\log q)^{n-1}}{n!}, 
\end{align*} 
where $\zeta(s, a)$ denotes the \emph{Hurwitz zeta function} and 
$B_n(x)$ is a \emph{Bernoulli polynomial} \citep[\cf \S 24.2, \S 25.11]{NISTHB}. 
For comparison, we may use the transformation identity in 
\eqref{eqn_nPowmTimesSeqfn_OGF_transform_stmt_v1} and an expansion of the 
exponential generating function for the \emph{Bernoulli numbers} in powers of $n$ 
to obtain that
\begin{align*} 
\sum_{n \geq 1} \frac{n^{\alpha} q^n}{1-q^n} & = 
     -(\log q) \times \sum_{n \geq 1} n^{\alpha-1} q^n - 
     \sum_{k \geq 0} \frac{(-1)^k B_k (\log q)^{k-1}}{k!} \times 
     \sum_{j=0}^{k+\alpha} \gkpSII{k+\alpha}{j} \frac{q^j \cdot j!}{(1-q)^{j+1}}. 
\end{align*} 
Relations of our new results to other open problems in number theory include the 
famous unresolved questions of whether there are infinitely-many perfect numbers, 
$n$ for which $\sigma(n) = 2n$, and whether there are any odd perfect numbers. 
It is beyond the scope of this article to give due attention to these famous 
unresolved problems, but we note that topics related to perfect numbers and their 
generalizations are an active research area for many 
researchers and mathematicians. 

\subsection{Generating functions and asymptotics} 

This section suggests several approaches to problems related to the 
generalized sum-of-divisors functions following as applications 
from our specific technique of an approach to these special function generating 
functions through J-fractions. 
First, we note that we may apply any number of known methods for obtaining 
asymptotic estimates of the coefficients in a formal power series to find new 
forms of asymptotic formulas for the generalized sum-of-divisors functions and the 
partial sums, $\Sigma_{\alpha}(x) := \sum_{n \leq x} \sigma_{\alpha}(n)$, 
by considering the generating functions, $\Conv_h(q, z)$ and $J_{\infty}(q, z)$, 
multiplied by a factor of $1 / (1-q)$. 
This consequence of using a generating-function-based technique to approaching the 
generalized divisor functions is particularly relevant since we have 
$h$-order accurate power series approximations to the generating functions of 
these special divisor functions which are rational in $q$ (and in $z$) for each 
finite case of $h \geq 2$. 
We compare the forms of these rational approximate generating functions to the 
known examples cited in Section \ref{subSection_ExamplsOfKnownGFs}. 

\subsection{Rational generating functions and new congruence results} 

One other consequence of utilizing the technique of using J-fraction expansions to 
generate the terms in our special Lambert series and $q$-series forms provides 
congruences for these functions modulo special functions of $q$. 
More precisely, when $h \geq 3$ for any divisor, 
$\widehat{d}(q)$ of the $h^{th}$ modulus, 
$M_h(a, b; q) := \lambda_1(a, b; q) \cdots \lambda_h(a, b; q)$, 
defined by \eqref{def_ParamSeqDefs_and_Notations} in 
Definition \ref{def_ParamSeqDefs_and_Notations}, we have that for $n < 2h$ 
\citep[\S 2]{FLAJOLET82} 
\begin{align*} 
J_{\infty}(a, b; q, z) & \equiv \Conv_h(a, b; q, z) && \pmod{M_h(a, b; q)} \\ 
     & \equiv \Conv_h(a, b; q, z) && \pmod{\widehat{d}(q)} \\ 
\frac{(a; q)_n}{(b; q)_n} & \equiv 
     [q^n] \Conv_h(a, b; q, z) && \pmod{q^{2h}}, 
\end{align*} 
and similarly for the special case series when $(a, b) := (q, q^2)$ and 
$n, x < 2h$ we have that 
\begin{align*} 
d(n) & \equiv 
     [q^n] \frac{\Conv_h(q, q^2; q, q)}{1-q} && \pmod{q^{2h}} \\ 
\Sigma_0(x) & = [q^x] \frac{\Conv_h(q, q^2; q, q)}{(1-q)^2} && \pmod{q^{2h}}. 
\end{align*} 
By considering the approximate generating functions, 
$\Conv_h(q, q^2; q, q) / (1-q)$, modulo $q^{2h}$ we are able to obtain 
simple rational functions in $q$ which are 
$(2h-1)$-order accurate in generating the divisor function: 
\begin{align} 
\label{footnote_ApproximateRationalGFExamples} 
\frac{1+4 q+8 q^2+11 q^3+10 q^4}{1+2 q+2 q^2-2 q^4} & = 
     1 + \sum_{n=1}^4 d(n) q^n + O(q^5) \\ 
\notag 
\frac{-1-5 q-14 q^2-29 q^3-46 q^4-62 q^5-71 q^6}{-1-3 q-6 q^2-8 q^3-7 q^4-4 q^5+q^6} & = 
1 + \sum_{n=1}^6 d(n) q^n + O(q^7) \\ 
\notag 
\textstyle{\frac{1+6 q+20 q^2+50 q^3+101 q^4+175 q^5+267 q^6+369 q^7+472 q^8}{1+4 q+10 q^2+19 q^3+29 q^4+37 q^5+40 q^6+38 q^7+32 q^8}} & = 
1 + \sum_{n=1}^7 d(n) q^n + O(q^8). 
\end{align} 
We can similarly form the $2h$-order accurate rational approximate generating 
functions for the sum-of-divisors function, $\sigma(n)$, by 
differentiating $qz \cdot \Conv_h(q, qz) / (1-q)$ with respect to $z$ 
which then implies the following analogous results: 
\begin{align*} 
\frac{q \left(1+3 q+3 q^2\right)}{(1-q) (1+q)} & = 
     \sum_{n=1}^3 \sigma(n) q^n + O(q^4) \\ 
\frac{q \left(1+7 q+25 q^2+62 q^3+115 q^4\right)}{\left(1+q+3 q^2\right) \left(1+3 q+3 q^2\right)} & = 
          \sum_{n=1}^5 \sigma(n) q^n + O(q^6) \\ 
\frac{q \left(1+9 q+44 q^2+155 q^3+430 q^4+998 q^5+2000 q^6\right)}{1+6 q+22 q^2+58 q^3+120 q^4+204 q^5+290 q^6+350 q^7} & = 
          \sum_{n=1}^7 \sigma(n) q^n + O(q^8) \\ 
\textstyle{\frac{q \left(1+11 q+65 q^2+276 q^3+935 q^4+2676 q^5+6696 q^6+14998 q^7+30592 q^8\right)}{1+8 q+37 q^2+126 q^3+347 q^4+812 q^5+1664 q^6+3050 q^7+5079 q^8+7776 q^9}} & = 
          \sum_{n=1}^9 \sigma(n) q^n + O(q^{10}). 
\end{align*} 
We can also reduce the coefficients of the numerator and denominator polynomials in 
$q$ in these rational generating function approximations modulo any integer 
$p \geq 2$ to form congruences for the sums-of-divisors functions modulo $p$. 
For example, when $h := 4, 5$ and $p := 5$, and when $1 \leq n < 2h$, we have that 
\begin{align*} 
\frac{q+4 q^2+4 q^3+3 q^6}{1+q+2 q^2+3 q^3+4 q^5} & \equiv 
     \sum_{n=1}^7 \left[\sigma(n)\pmod{5}\right] q^n + O(q^8) \\ 
\frac{q+q^2+q^4+q^6+q^7+3 q^8+2 q^9}{1+3 q+2 q^2+q^3+2 q^4+2 q^5+4 q^6+4 q^8+q^9} & \equiv 
     \sum_{n=1}^9 \left[\sigma(n)\pmod{5}\right] q^n + O(q^{10}) 
\end{align*} 
The identities in the previous several equations can be generalized to enumerate 
$\sigma_{\alpha}(n)$ and $\Sigma_{\alpha}(x)$ for $\alpha \in \mathbb{Z}^{+}$ 
by first differentiating with respect to $z$ and then setting $z \mapsto q$ 
as in the generating function constructions from 
Section \ref{subSection_ModifiedGFs_for_IntOrderDivFns}. 

\section{Conclusions} 
\label{Section_Concl} 

\subsection{Summary} 

We have defined the forms of infinite J-fractions in the form of 
\eqref{eqn_J-Fraction_Expansions} 
whose power series expansions in $z$ generate the ratio of 
$q$-Pochhammer symbols, $(a; q)_n / (b; q)_n$, for all $n \geq 0$. 
We focused on the special case of these expansions where $(a, b) := (q, q^2)$, and 
subsequently proved the forms of new convergent infinite $q$-series sums 
involving the generalized Stirling $q$-coefficients we defined in 
\eqref{eqn_GenStirqCoeffs_GenS1hkc} that enumerate the divisor function, 
$d(n)$, and the generalized sums-of-divisors functions, 
$\sigma_{\alpha}(n)$ for $\alpha \in \mathbb{Z}^{+}$. 
We cite comparisons of these results generating the divisor function and the positive
integer-order sums-of-divisors functions 
according to the expansions of \eqref{eqn_Qim1Qi_ConvFnSum_Denom_Exp} with the 
known generating function expansions related to these special functions 
cited in Section \ref{subSection_ExamplsOfKnownGFs} of the introduction. 
Further applications of the results we have proved in this article are derived 
from the subtlety that most of the expansions given for the convergent 
denominator functions, $Q_h(z)$, involved in the infinite sums for these 
generalized J-fractions are stated and proved in the general case of the 
implicit sequences in \eqref{eqn_J-Fraction_Expansions}. 

\renewcommand{\refname}{References} 

\nocite{HARDYWRIGHTNUMT}

\end{document}